\documentclass[12pt]{article}
\usepackage{amssymb}

\newtheorem{thm}{Theorem}

\newtheorem{theorem}{Theorem}[section]
\newtheorem{corollary}[theorem]{Corollary}

\newtheorem{lemma}[theorem]{Lemma}

\newenvironment{proof}[1][Proof]%
{\par\addvspace{6pt}\noindent{\bf #1.}\hskip\labelsep\ignorespaces}%
{{\hfill $\square$}\par\addvspace{6pt}}

\def\card#1{\vert #1 \vert}
\def\gpindex#1#2{\card {#1\colon #2}}
\def\irr#1{{\rm  Irr}(#1)}

\def\cd#1{{\rm  cd}(#1)}
\def\cent#1#2{{\bf C}_{#1}(#2)}
\def\gpcen#1{{\bf Z} (#1)}
\def\ker#1{{\rm ker} (#1)}

\newcommand{\nl}[1]{{\rm nl} (#1)}

\begin{document}
%

\title{The vanishing-off subgroup}

\author {
       Mark L.\ Lewis
    \\ {\it Department of Mathematical Sciences, Kent State University}
    \\ {\it Kent, Ohio 44242}
    \\ E-mail: lewis@math.kent.edu
       }
\date{June 20, 2008}

\maketitle

\begin{abstract}
In this paper, we define the vanishing-off subgroup of a nonabelian
group. We study the structure of the quotient of this subgroup and a
central series obtained from this subgroup.

MSC primary: 20C15, MSC secondary: 20D25
\end{abstract}

\section{Introduction}

Throughout this paper, $G$ is a finite group and $\irr G$ is the set
of irreducible characters of $G$.  Following Chapter 12 of
\cite{text}, we define the {\it vanishing-off subgroup} $V (\chi)$
of a character $\chi$ to be the subgroup generated by the elements
of $G$ where $\chi$ is not $0$.  Mathematically, we write $V (\chi)
= \langle g \in G \mid \chi (g) \ne 0 \rangle$.  It is not difficult
to see that $V (\chi)$ is a normal subgroup of $G$, and if $H$ is
any subgroup so that $\chi$ vanishes on $G \setminus H$, then $V
(\chi) \le H$. In other words, $V (\chi)$ is the smallest subgroup
$V \le G$ so that $\chi$ vanishes on $G \setminus V$.

When $\lambda$ is a linear character of $G$, we know that $\lambda$
never vanishes on $G$, so $V (\lambda) = G$.  With this in mind, one
should only look at $V (\chi)$ when $\chi$ is not linear.  In fact,
if $G$ is a nonabelian group, we define the {\it vanishing-off
subgroup $V (G)$ of $G$} to be the subgroup of $G$ generated by the
elements $g \in G$ where there is some nonlinear character $\chi \in
\irr G$ so that $\chi (g) \ne 0$.  We write $\nl G$ for the
nonlinear characters in $\irr G$.  With this in mind, we have
$$
V (G) = \langle g \in G \mid {\rm ~there~exists~} \chi \in \nl G
{\rm ~such~that~} \chi (g) \ne 0 \rangle.
$$
Notice that $V (G)$ will be a characteristic subgroup of $G$.  Every
character in $\nl G$ will vanish on $G \setminus V(G)$, and $V (G)$
is the smallest subgroup $V \le G$ so that every character in $\nl
G$ vanishes on $G \setminus V$.

It is often not the case that $V (G)$ is a proper subgroup of $G$.
Using the remarks in Section 12 of \cite{text}, we deduce that
$\gpindex G{V (G)}$ will divide $\chi (1)^2$ for all $\chi \in \nl
G$. It follows that all of the primes that divide $\gpindex G{V
(G)}$ will divide all the nonlinear degrees in $\cd G$.  Hence, if
$G$ has characters in $\nl G$ of coprime degrees, then $G = V (G)$.
In \cite{complete}, we along with Bianchi, Chillag, and Pacifici
showed that every nonsolvable group will have characters in $\nl G$
of coprime degree.  Therefore, $G = V (G)$ when $G$ is nonsolvable.
We will also provide examples where $G = V (G)$ when $G$ is solvable
and even nilpotent.

However, there are many examples where $V (G) < G$. The most
prominent examples are the Camina groups that have been studied in
many places including \cite{camina}, \cite{coprime}, \cite{MacD1},
and \cite{Yong}.  In fact, the condition for being a Camina group
can be stated in terms of $V (G)$.  In particular, a group $G$ will
be a Camina group if and only if $V (G) = G'$.  We should note that
this includes both the Frobenius groups with abelian Frobenius
complements and extra-special groups.  In \cite{GCG}, we generalized
the definition of a Camina group. A group $G$ will be a generalized
Camina group if and only if $V (G) = \gpcen G G'$.

In this paper, we study some of the properties of $V (G)$.  The
first easy fact that we prove is the following:

\begin{thm}
Let $G$ be a nonabelian, solvable group.  Then $G/V (G)$ is either
cyclic or an elementary abelian $p$-group for some prime $p$.
\end{thm}

To obtain the deeper results regarding the vanishing-off subgroup of
$G$, we look at a central series in terms of $V (G)$.  We set $V_1
(G) = V (G)$, and for $i \ge 2$, we set $V_i (G) = [V_{i-1} (G),
G]$.  We will compare this series with the lower central series for
$G$.  We set the following notation for the lower central series of
$G$. We set $G_1 = G$, and for $i \ge 2$, we set $G_i =
[G_{i-1},G]$. Note that $G_2 = G'$.  We show that $V_i (G) \le G_i$,
and we are interested in the case when $V_i (G) < G_i$ for some $i
\ge 2$. In this next theorem, we study the case when $V_2 (G) <
G_2$.  We obtain information regarding the terms of this series.  In
particular, we will show that all of the quotients $G_i/V_i (G)$
will be elementary abelian $p$-groups for some prime $p$.  In
addition, we will show that $G/V_2 (G)$ will be what we called a
VZ-group in \cite{notes} and \cite{GCG} and those are the class of
groups studied extensively in \cite{extreme}.

\begin{thm}
Let $G$ be a group.  Then $G_{i+1} \le V_i (G) \le G_i$ for all $i
\ge 1$.  If in addition $V_2 (G) < G_2$, then the following hold:

\begin{enumerate}
\item There is a prime $p$ so that $G_i/V_i (G)$ is an elementary
abelian $p$-group for all $i \ge 1$.
\item There exist positive integers $m \le n$ so that $\gpindex G{V
(G)} = p^{2n}$ and $\gpindex {G_2}{V_2 (G)} = p^m$, and $\cd {G/V_2
(G)} = \{ 1, p^n \}$.
\end{enumerate}
\end{thm}

When $V_3 (G) < G_3$, we can obtain further information regarding
the structure of the group.  The careful reader will recognize that
many of the conclusions of this result are similar to the
conclusions that we obtained for generalized Camina groups of
nilpotence class $3$ in \cite{GCG} (and many of the results there
are probably corollaries of the result here).

\begin{thm} \label{main3}
Suppose $G$ is a group where $V_3 (G) < G_3$.  Let $Z/V_3 (G) =
\gpcen {G/V_3 (G)}$ and $C/V_3 (G) = \cent {G/V_3 (G)}{G'/V_3 (G)}$.
Then the following are true:
\begin{enumerate}
\item $\gpindex G{V_1 (G)} = \gpindex {G'}{V_2 (G)}^2$.
\item $V_2 (G) = Z \cap G'$.
\item Either $\gpindex GC = \gpindex {G'}{V_2 (G)}$ or $C = V_1 (G)$.
\item $V_1 (G) \le V (C)$.
\item If $V_1 (G) < C$, then $C' = V_2 (G)$.
\item If $V_1 (G) < C$ and $[V_1 (G),C] < V_2 (G)$, then $\gpindex GC$ is a square.
\end{enumerate}
\end{thm}

We should note that there is some important variances in this result
from the results of \cite{GCG}.  In particular, we were able to
prove that $\gpindex GC = \gpindex {G'}{V_2 (G)}$ is a square
whenever $G$ is a generalized Camina group of nilpotence class $3$.
Here, we see that either $\gpindex GC = \gpindex {G'}{V_2 (G)}$ or
$C = V_1 (G)$.  We will show that both cases can occur.  When $V_1
(G) < C$, we need an additional assumption to obtain the conclusion
that $\gpindex GC$ is a square. This additional assumption is
actually necessary, as we have examples where $\gpindex GC$ is not a
square.

Following \cite{GCG}, it is tempting to conjecture that $\gpindex
{G_3}{V_3 (G)} \le \gpindex {G'}{V_2 (G)}$, however, this need not
be true since we have an example where $\gpindex {G_3}{V_3 (G)} =
\gpindex G{V_1 (G)} = 4$ and $\gpindex {G'}{V_2 (G)} = 2$.  In any
case, we would like to obtain some bound for $\gpindex {G_i}{V_i
(G)}$ in terms of $\gpindex G{V (G)}$ when $i \ge 3$.

One important fact regarding nilpotent, generalized Camina groups is
that their nilpotence class is at most $3$, and it is at most $2$ if
the associated prime is $2$.  With this in mind, we define the {\it
vanishing height} of $G$ to be the largest value of $i$ so that $V_i
(G) < G_i$.  It would be tempting to conjecture that $i \le 3$;
however, we have examples where $i = 4$, and it seems likely that
there is no bound on $i$.

If $G$ is nilpotent with nilpotence class $c$, we will show that
$Z_{c-1} \le V (G)$ where $Z_{c-1}$ is the $(c-1)$st term in the
upper central series for $G$.  We would like to determine if there
is a largest nilpotence class $c$ so that $V (G) = Z_{c-1}$.  We
have an example with $c = 4$.  Notice that if $G$ satisfies $V (G) =
Z_{c-1}$, then $G$ will have vanishing height $c$.


\section{Generalized Camina pairs}

Let $N$ be a normal subgroup of a group $G$.  Following the
literature, we say that $(G,N)$ is a Camina pair if every element $g
\in G \setminus N$ is conjugate to all of $gN$.  These pairs were
first studied in \cite{camina}.  In this section, we will define a
similar pair called a generalized Camina pair.  We will show that
generalized Camina pairs are useful in studying the vanishing-off
subgroup.  We begin with the following general lemma which is proved
in \cite{GCG}.  (This result was suggested by work in \cite{camina}
and \cite{ChMc}.)

\begin{lemma}\label{facts}
Let $g$ be an element of a group $G$.  Then the following are
equivalent:
\begin{enumerate}
\item The conjugacy class of $g$ is $gG'$.
\item $|\cent Gg| = \gpindex G{G'}$.
\item For every $z \in G'$, there is an element $y \in G$ so that
$[g,y] = z$.
\item $\chi (g) = 0$ for all nonlinear $\chi \in \irr G$.
\end{enumerate}
\end{lemma}

We will say $g \in G$ is a {\it Camina element} of $G$ if $g$
satisfies one of the equivalent conditions of Lemma \ref{facts}.

\begin{lemma}\label{Camelt}
Let $N$ be a subgroup of a group $G$.  If $N$ contains a Camina
element for $G$, then $G' = [G,N]$.
\end{lemma}

\begin{proof}
Let $g \in N$ be a Camina element for $G$.  Then $G' = [g,G]$ by
Lemma \ref{facts} and $[g,G] \le [N,G] \le G'$.  The conclusion then
follows.
\end{proof}

Let $N$ be a normal subgroup of $G$.  We say that $(G,N)$ is a {\it
generalized Camina pair} (abbreviated GCP) if every element in $G
\setminus N$ is a Camina element for $G$.  (We mention here that we
allow $N = G$, although $(G,G)$ is not a very interesting GCP.)

We note that GCPs are related to an idea defined in \cite{weak}.
Following \cite{weak}, we define $(G,N,M)$ to be a {\it Camina
triple} if for every element $g \in G \setminus N$, then $g$ is
conjugate to all of $gM$. Notice that $(G,N,N)$ is a Camina triple
if and only if $(G,N)$ is a Camina pair, and $(G,N,G')$ is a Camina
triple if and only if $(G,N)$ is a GCP.

At this time, we will gather facts about generalized Camina pairs.
{}From Lemma \ref{facts}, we have a number of equivalent conditions
for a GCP.

\begin{corollary} \label{GCP1}
Let $1 < N$ be a normal subgroup of a nonabelian group $G$.  The
following are equivalent:
\begin{enumerate}
\item $(G,N)$ is a GCP.
\item $|\cent Gg| = \gpindex G{G'}$ for all $g \in G \setminus N$.
\item Every character in $\nl G$ vanishes on $G \setminus N$.
\item For every element $g \in G \setminus N$ and every element $z
\in G'$, there is an element $y \in G$ so that $[g,y] = z$.
\end{enumerate}
\end{corollary}

\begin{proof}
We first suppose that $(G,N)$ is a GCP.  Thus, for every $g \in G
\setminus N$, then ${\rm cl} (g) = gG'$.  By Lemma \ref{facts}, this
is equivalent to $\card {\cent Gg} = \gpindex G{G'}$ for every $g
\in G \setminus N$.  Also, by Lemma \ref{facts}, this is equivalent
to every nonlinear character in $\irr G$ vanishing on $G$.  Finally,
by Lemma \ref{facts}, this is equivalent to having for element $z
\in G'$ some element $y \in G$ so that $[g,y] = z$.
\end{proof}

We now prove a number of basic facts regarding GCPs.

\begin{lemma}\label{basics}
If $(G,N)$ is a GCP, then the following are true:
\begin{enumerate}
\item $G' \le N$.
\item If $N \le M < G$, then $(G,M)$ is a GCP.
\item If $K$ is a normal subgroup of $G$ satisfying $K \le N$, then
$(G/K,N/K)$ is a GCP.
\item If $G$ is nonabelian, then $\gpcen G \le N$.
\end{enumerate}
\end{lemma}

\begin{proof}
Suppose there exists $g \in G' \setminus N$.  Then by definition of
GCP, $g$ is conjugate to all of $gG' = G'$  This implies that $g$ is
conjugate to $1$, but this is not possible since $g \not\in N$
implies $g \ne 1$.  Thus, we have $G' \le N$.  Suppose $1 \le G' \le
N \le M < G$, and so, $M$ is a normal subgroup of $G$. If $g \in G
\setminus M$, then $g \in G \setminus N$, and so, ${\rm cl} (g) =
gG'$.  We conclude that $(G,M)$ is a GCP.  Now, suppose that $K$ is
a normal subgroup of $G$ so that $K \le N$.  It follows that $N/K$
is a normal subgroup of $G/K$. Since every nonlinear irreducible
character of $G$ vanishes on $G \setminus N$, it follows that every
nonlinear character of $G/K$ will vanish on $G/K \setminus N/K$.
Thus, $(G/K, N/K)$ is a GCP.  Finally, suppose that $G$ is
nonabelian.  If $g \in \gpcen G$, then $g$ is conjugate only to
itself and not to $gG'$ since $G' > 1$. This implies that $g \in N$.
\end{proof}

We next consider abelian groups and GCPs.

\begin{lemma}
Let $G$ be a group.  Then $G$ is abelian if and only if $(G,1)$ is
GCP.
\end{lemma}

\begin{proof}
If $(G,1)$ is a GCP, then $G' = 1$ by Lemma \ref{basics}, and so,
$G$ is abelian.  On the other hand, if $G$ is abelian, then $(G,1)$
is a GCP, and the result is proved.
\end{proof}

When $G$ is nilpotent and $(G,N)$ is a GCP with $N < G$, we see that
$G$ is essentially a $p$-group for some prime $p$.

\begin{lemma}
Suppose $(G,N)$ is a GCP where $G$ is nonabelian and nilpotent and
$N < G$. Then $G/N$ is a $p$-group for some prime $p$ and $G = P
\times Q$ where $P$ is a $p$-group and $Q$ is an abelian $p'$-group.
\end{lemma}

\begin{proof}
Let $p$ be a prime divisor of $\gpindex GN$, and consider $\chi \in
\nl G$.  By Corollary \ref{GCP1}, we know that $\chi$ vanishes on $G
\setminus N$.  This implies that $\gpindex GN$ divides $\chi (1)^2$
(see page 200 in \cite{text}).  Hence, $p$ divides every nonlinear
degree in $\cd G$.  If we take $G = P \times Q$ where $P$ is the
Sylow $p$-subgroup of $G$ and $Q$ is the Hall $p$-complement, then
every character in $\irr Q$ is linear. Therefore, $Q$ is abelian.
\end{proof}

\section{The subgroup $V(G)$}

In this section, we study some of the basic properties of the
characteristic subgroup $V(G)$.  We begin with a result comparing
two GCPs.

\begin{lemma} \label{intersection}
If $(G,N_1)$ and $(G,N_2)$ are GCPs, then $(G,N_1 \cap N_2)$ is a
GCP.
\end{lemma}

\begin{proof}
We know that $1 \le G' \le N_1 \cap N_2$.  If $g \in G \setminus
(N_1 \cap N_2)$, then either $g \in G \setminus N_1$ or $g \in G
\setminus N_2$.  In either case, we have ${\rm cl} (g) = gG'$.  We
conclude that $(G,N_1 \cap N_2)$ is a GCP.
\end{proof}

Let $G$ be a group, and let ${\cal N} (G) = \{ N \le G \mid (G,N)
~{\rm is~a~GCP} \}$.  We can now enumerate some of the properties of
$V (G)$.

\begin{lemma} \label{V(G)}
Let $G$ be a nonabelian group.  Then the following are true:

\begin{enumerate}
\item $(G, V (G))$ is a GCP.
\item $G' \le V (G)$.
\item $\gpcen G \le V(G)$.
\item Every element of $G \setminus V(G)$ is a Camina element for
$G$.
\item If $(G,N)$ is a GCP, then $V (G) \le N$.
\item $V (G) = \cap_{N \in {\cal N} (G)} N$.
\item $V (G) = \prod_{\chi \in \nl G} V (\chi)$.
\item $V (G)$ is a characteristic subgroup of $G$.
\item If $V (G) < M$, then $[G,M] = G'$.
\end{enumerate}
\end{lemma}

\begin{proof}
Let $(G,N)$ be a GCP.  By Corollary \ref{GCP1}, we know that every
character in $\nl G$ vanishes on $G \setminus N$ and so, $V (G) \le
N$.  It follows that $V (G) \le \cap_{N \in {N \in {\cal N} (G)}}
N$.  The fact that $(G, V (G))$ is a GCP is an application of
Corollary \ref{GCP1}. Thus, $V (G) \in {\cal N} (G)$, and we
conclude that $V (G) = \cap_{N \in {\cal N} (G)} N$.  The facts that
$G' \le V (G)$, $\gpcen G \le V (G)$, every element of $G \setminus
V(G)$ is a Camina element, and every nonlinear character in $\irr G$
vanishes on $G \setminus V (G)$ are consequences of Lemma
\ref{basics}. Notice that $V (\chi) \le V (G)$ for all $\chi \in \nl
G$.  Thus, $\prod_{\chi \in \nl G} V (\chi) \le V (G)$.  On the
other hand, every character in $\nl G$ will vanish on $G \setminus
\prod_{\chi \in \nl G} V (\chi)$.  This implies $V (G) \le
\prod_{\chi \in \nl G} V (\chi)$, and hence, we get $V (G) =
\prod_{\chi \in \nl G} V (\chi)$. Any automorphism of $G$ will just
permute the elements of ${\cal N} (G)$, so $V (G)$ will be
characteristic. Finally, if $V (G) < M$, then $M$ must contain a
Camina element of $G$, and so $[G,M] = G'$ by Lemma \ref{Camelt}.
\end{proof}

We now describe relationship between $V (G)$ and nonabelian
quotients of $G$.

\begin{lemma}\label{quots}
Let $N$ be a normal subgroup of $G$ so that $G/N$ is nonabelian.
Then $N \le V(G)$ and $V(G/N) \le V (G)/N$.
\end{lemma}

\begin{proof}
We can find a character $\chi \in \nl {G/N}$.  Now, observe that $N
\le \ker {\chi} \le V (\chi) \le V (G)$. By Lemma \ref{V(G)}, we
know that $(G,V (G))$ is a GCP, and thus, by Lemma \ref{basics},
$(G/N,V(G)/N)$ is a GCP. Applying Lemma \ref{V(G)}, we deduce that
$V (G/N) \le V(G)/N$.
\end{proof}

We describe the structure of $G/ V(G)$ when $G$ is nonabelian. When
$G$ has a quotient that is nonabelian and nilpotent, $G/V(G)$ will
be an elementary abelian $p$-group for some prime $p$.

\begin{lemma}
Let $G$ be a group with a nonabelian nilpotent quotient.  Then $G/
V(G)$ is an elementary abelian $p$-group for some prime $p$.
\end{lemma}

\begin{proof}
Let $K$ be maximal among normal subgroups of $G$ that have a
nonabelian solvable quotient.  Since $G$ has a nonabelian nilpotent
quotient, we may assume that $G/K$ is nilpotent.  We know by Lemma
12.3 of \cite{text} that $G/K$ is a $p$-group for some prime $p$.
Consider a character $\chi \in \nl {G/K}$.  It is not difficult to
see that $V (\chi)/K = \gpcen {G/K}$, and so, $G/ V(\chi)$ is an
elementary abelian $p$-group by Lemma 12.3 (a) of \cite{text}.  By
Lemma \ref{V(G)}, we have $V (\chi) \le V(G)$, and so, $G/V(G)$ is a
quotient of $G/V(\chi)$.
\end{proof}

When $G$ has a quotient that is a Frobenius group with an abelian
Frobenius complement, we show that $G/V(G)$ is cyclic.

\begin{lemma}
Suppose $G$ has a quotient that is a Frobenius group with an abelian
Frobenius complement. Then $G/V (G)$ is cyclic.
\end{lemma}

\begin{proof}
We now suppose that $K$ is a normal subgroup where $G/K$ is a
Frobenius group with Frobenius kernel $N/K$.  Without loss of
generality, we may assume that $N/K$ is a chief factor for $G$. Fix
$\chi \in \nl {G/K}$.  We know that $\chi$ is induced from $N$, so
$K \le \ker {\chi} \le V (\chi) \le N$ (see Chapter 12 of
\cite{text}). Since an irreducible character cannot be a multiple of
the the regular character of some quotient, it follows that $\ker
{\chi} < V (\chi)$.  Since $K/N$ is a chief factor, this implies
that $K = \ker {\chi}$ and $V (\chi) = N$. Now, $G/N$ is isomorphic
to an abelian Frobenius complement, and so, $G/V (\chi)$ is cyclic.
Since $V (\chi) \le V (G)$ by Lemma \ref{V(G)}, this gives the
result.
\end{proof}

Using Chapter 12 of \cite{text}, we know that if $G$ is a nonabelian
solvable group, then it must have a nonabelian quotient that is
either nilpotent or a Frobenius group with an abelian Frobenius
complement. Thus, we may combine the two previous lemmas to
determine the structure of $G/V (G)$ when $G$ is nonabelian and
solvable.

\begin{corollary}
Suppose $G$ is a nonabelian solvable group.  Then either $G/ V(G)$
is an elementary abelian $p$-group for some prime $p$ or $G/ V(G)$
is cyclic.
\end{corollary}

\begin{proof}
Take $K$ to be maximal subject to $K$ being normal and $G/K$ is
nonabelian.  By Lemma 12.3 of \cite{text}, $G/K$ is either a
nonabelian $p$-group or $G/K$ is a Frobenius group with an abelian
Frobenius complement.  Thus, the previous two lemmas apply and yield
the result.
\end{proof}

%
%
%

We now consider the terms of the upper central series of $G$. Let
$Z_1 = \gpcen G$ and $Z_i/Z_{i-1} = \gpcen {G/Z_{i-1}}$ for $i > 1$.
We begin with an inductive step.

\begin{lemma}\label{indstep}
Suppose $G$ is a group so that $G' \not\le Z_m$, then $Z_{m+1} \le V
(G)$.
\end{lemma}

\begin{proof}
We will prove this lemma by induction.  Our inductive hypothesis is
that $Z_i (G) \le V (G)$ for all $i$ with $1 \le i \le m+1$, and we
work by induction on $i$. The base case ($i=1$) is Lemma \ref{V(G)}.
For the inductive case ($i > 1$), we assume that $Z_{i-1} \le V
(G)$. Since $i \le m$, we see that $G' \not\le Z_{i-1}$. By Lemma
\ref{basics} (3), $(G/Z_{i-1}, V (G)/Z_{i-1})$ is a GCP. Applying
Lemma \ref{basics} (4) to $G/Z_{i-1}$, we obtain $\gpcen {G/Z_{i-1}}
\le V (G)/Z_{i-1}$, and so, $Z_i \le V (G)$. This proves the
inductive step, and hence the lemma.
\end{proof}

These next two lemmas show that $V (G)$ must be relatively large.
First, when $G$ is nilpotent with class $c$, we show that $V (G)$
must contain $Z_{c-1}$.

\begin{corollary} \label{GCP2}
Suppose $G$ is nilpotent of nilpotence class $c$. Then $Z_{c-1} \le
V(G)$.
\end{corollary}

\begin{proof}
We know that $G' \not\le Z_{c-2} $, so by Lemma \ref{indstep}, we
have $Z_{c-1} \le V (G)$.
\end{proof}

Also, when $G$ is not nilpotent we get a similar result. We set
$Z_\infty = Z_n$ when for some integer $n$ we have $Z_n = Z_m$ for
all $m \ge n$.  Often $Z_{\infty}$ is called the hyper-center of
$G$.  We show that if $G$ is not nilpotent, then $V (G)$ must
contain the hyper-center of $G$.

\begin{corollary}
Suppose $G$ is not nilpotent.  Then $Z_{\infty} \le V (G)$.
\end{corollary}

\begin{proof}
Let $n$ be the integer so that $Z_n  = Z_{\infty} $.  We know that
$G' \not\le Z_n$, and so, $Z_{n+1} \le V (G)$.  Since $Z_{n+1} =
Z_{\infty}$, we have the desired conclusion.
\end{proof}

\section{The central series associated with V(G)} \label{cla3}

In \cite{GCG}, we studied many of the properties of generalized
Camina groups.  We will show that many of the properties of
generalized Camina groups can be found in general (nilpotent) groups
using the vanishing-off subgroup. Much of the work in this section
is based on the work in Section 5 of \cite{MacD1}.

Our goal in this section is to prove Theorem \ref{main3}. We prove
Theorem \ref{main3} in a series of lemmas.  In these lemmas, we will
study the following hypothesis.  When the parent group is clear, we
will write $V_i$ in place of $V_i (G)$.  We begin with some
immediate observations.

\begin{lemma} \label{cl3}
Let $G$ be a group.  Then the following facts are true.

\begin{enumerate}
\item For all $i \ge 1$, then $G_{i+1} \le V_i \le G_i$.
\item If $V_n < G_n$ for some $n$, then $V_i < G_i$ for all $i$ with
$1 \le i \le n$.
\item If $V_n < G_n$, then $G/V_n$ is nilpotent of nilpotence class
$n$ and $V_i (G/V_n) = V_i/V_n$ for $1 \le i \le n$.
\end{enumerate}
\end{lemma}

\begin{proof}
We work by induction on $i$.  By Lemma \ref{basics}, $G' = G_2 \le
V_1 \le G_1$, and this gives the base case $i = 1$. We now suppose
that $i > 1$, so that $G_i \le V_{i-1} \le G_{i-1}$. this yields
$G_{i+1} = [G_i,G] \le [V_{i-1},G] = V_i$ and $V_i = [V_{i-1},G] \le
[G_{i-1},G] = G_i$, and (1) is proved.  Suppose $V_i = G_i$ for some
$i$.  Then using induction one can show that $V_j = G_j$ for all $j
\ge i$.  Hence, if $i \le n$ and $V_n < G_n$, then $V_i < G_i$.

Notice that $G_n \le V_{n-1}$ and $V_{n-1}/V_n$ is central in
$G/V_n$.  Hence, $G_n/V_n$ is central in $G/V_n$ and nontrivial.
Thus, $G/V_n$ has nilpotence class $n$.

To prove that $V_i (G/V_n) = V_i/V_n$, we work by induction on $i$.
We start with $i = 1$.  By Lemma \ref{quots}, we know that $V_1
(G/V_n) = V_1/V_n$. Suppose now that $i > 1$, and $V_{i-1} (G/V_n) =
V_{i-1}/V_n$.  We know that $V_i (G/V_n) = [V_{i-1} (G/V_n),G/V_n] =
[V_{i-1}/V_n,G/V_n] = [V_{i-1},G]V_n/V_n = V_i/V_n$ which is the
desired result.
\end{proof}

We now consider the structure of $G/V_2$ when $V_2 < G_2 = G'$.
First, we need some notation.  As in \cite{notes} and \cite{GCG}, a
group $G$ is a VZ-group if all its nonlinear irreducible characters
vanish off of the center.  Hence, $G$ is a VZ-group if and only if
$V (G) = \gpcen G$.  It is easy to see in a VZ-group that all
nonlinear characters will be fully-ramified with respect to the
center.  Many results regarding VZ-groups can be found in
\cite{extreme}.  In Lemma 2.4 of \cite{GCG}, we noted the following
fact:

\begin{lemma} \label{main1} [\cite{GCG}]
Let $G$ be a VZ-group, then $G/\gpcen G$ and $G'$ are elementary
abelian $p$-groups for some prime $p$.  Furthermore, there exist
positive integers $m \le n$ so that $\card {G'} = p^m$ and $\gpindex
G{\gpcen G} = p^{2n}$.  In addition, $\cd G = \{ 1, p^n \}$.
\end{lemma}

When $V_2 < G'$, we will see that $G/V_2$ has the structure of a
VZ-group.  Notice that this implies that if  $\gpindex G{V_1}$ is
not a square, then $\gpindex {G'}{V_2} = 1$.

\begin{lemma}\label{VZquo}
Let $G$ be a group so that $V_2 < G' = G_2$.  Then $G/V_2$ is a
VZ-group with $V (G/V_2) = V_1/V_2$.  In particular, $G/ V_1$ and
$G'/V_2$ are elementary abelian $p$-groups for some prime $p$, and
there exist positive integers $m \le n$ so that $\gpindex G{V_1} =
p^{2n}$ and $\gpindex {G'}{V_2} = p^m$.  In addition, $\cd {G/{V_2}}
= \{ 1, p^n \}$.
\end{lemma}

\begin{proof}
Let $V = V_1$ and $W = V_2$.  Since $W < G'$, it follows that $V <
G$ and $G/W$ is a nonabelian group. By Lemma \ref{cl3}, $V/W =
V(G/W)$. It is easy to see that $V/W$ is contained in $\gpcen
{G/W}$, and by Lemma \ref{V(G)}, $V(G/W)$ contains $\gpcen {G/W}$.
Therefore, $V (G/W) = \gpcen {G/W}$, and $G/W$ is a VZ-group. The
remaining results follow from Lemma \ref{main1}.
\end{proof}

We now assume that $V_2 < G' = G_2$.  Following Lemma \ref{VZquo},
we write $\gpindex G{V_1} = p^{2n}$ and $\gpindex {G'}{V_2} = p^m$
(where $m \le n$).

\begin{lemma} \label{Fact 4}
Assume $V_2 < G_2$.  Then $G_i/V_i$ is an elementary abelian
$p$-group for all $i \ge 1$ for some prime $p$.
\end{lemma}

\begin{proof}
We work by induction on $i$.  If $i = 1, 2$, then this is Lemma \ref
{VZquo}.  Thus, we may assume that $i \ge 3$.  Also, by the
inductive hypothesis, we have that $G_{i-1}/V_{i-1}$ is an
elementary abelian $p$-group.  Without loss of generality, we may
assume that $V_i = 1$. Notice that this implies that $V_{i-1}$ is
central in $G$.  We have $G_i \le V_{i-1}$, and so, $G_i$ is central
in $G$. As $G_i$ is abelian, it suffices to show that its generators
all have order $p$. Recall that $G_i$ is generated by elements of
the form $[x,w]$ where $x \in G_{i-1}$ and $w \in G$.  By inductive
hypothesis, we know that $x^p \in V_{i-1}$. Since $G_i$ and
$V_{i-1}$ are central in $G$, we have $[x,w]^p = [x^p,w] = 1$.  This
implies that the generators of $G_i$ all have order $p$, and the
result follows.
\end{proof}

We compute the index of an important centralizer in this next lemma.

\begin{lemma} \label{Fact 7}
Assume $V_2 < G_2$.  Suppose $a \in G \setminus V_1$ and $A/V_2 =
\cent {G/V_2}{aV_2}$. Then $\gpindex GA = \gpindex {G'}{V_2} = p^m$.
\end{lemma}

\begin{proof}
Consider the map $x \mapsto [a,x]V_2$ from $G$ to $G'/V_2$.  We know
that $G' \le V_1$ and $V_1/V_2$ is central in $G/V_2$, so $G'/V_2$
is central in $G/V_2$.  Since $G'/V_2 \le \gpcen {G/V_2}$, it
follows that this map is a homomorphism, and it is easy to see that
its kernel is $A$. Since $a \not\in V_1$ and $(G,V_1)$ is a GCP,
this homomorphism will be onto, and the result is just the first
isomorphism theorem.
\end{proof}

For the rest of this section, we will consider the hypothesis that
$V_3 < G_3$.  Suppose that $G$ has nilpotence class $3$ and $V_1 =
Z_2 (G)$, then $V_2 \le \gpcen G$, and hence, $V_3 = [V_2,G] = 1 <
G_3$. Therefore, if we have a group where all the nonlinear
irreducible characters vanish off of the second center, then the
group $G$ will satisfy $V_3 < G_3$.

For the remainder of this section, we set $C/V_3 = \cent
{G/V_3}{G'/V_3}$. We now consider the elements in $G \setminus V_1$.
In this next lemma, we use the Jacobi-Witt identity which says that
if $G$ has nilpotence class $3$ and $w,x,y \in G$, then $[w,x,y]
[x,y,w] [y,w,x] = 1.$

\begin{lemma} \label{Fact 6}
Assume $V_3 < G_3$.  Suppose $b \in C \setminus V_1$ and let $B/V_2
= \cent {G/V_2}{bV_2}$.  Then $B \le C$.
\end{lemma}

\begin{proof}
We may assume without loss of generality that $V_3 = 1$.  Under this
hypothesis, we have $C = \cent G{G'}$.  Also, we see that $G$ has
nilpotence class $3$ and $V_2$ is central in $G$. Fix $x \in B$, and
let $w \in G$ be arbitrary. Note that $B = \{ g \in G \mid [b,g] \in
V_2 \}$. Thus, $[b,x] \in V_2 \le Z_1$, so $[b,x,w] = 1$. Also,
$[x,w] \in G'$ and $b \in C$, so $[x,w,b] = 1$. Applying the
Jacobi-Witt identity, we have $[w,b,x] = 1$.  Since $b \not\in V_1$
and $(G,V_1)$ is a GCP, we know that $[w,b]$ runs through all of
$G'$ as $w$ runs through $G$. This implies that $x$ centralizes
$G'$, and hence, $x \in C$.
\end{proof}

The following lemma is important to understand the structure of $G$.

\begin{lemma} \label{Fact 8}
Assume $V_3 < G_3$.  Suppose $a \in G \setminus C$ and $A/V_2 =
\cent {G/V_2}{aV_2}$. Then $A \cap C = V_1$.
\end{lemma}

\begin{proof}
Without loss of generality, we may assume that $V_3 = 1$, and hence,
$C = \cent G{G'}$. We first prove that $V_1 \le C$. To see this
consider $[G',V_1] = [G,G,V_1]$. Observe that $[G,V_1,G] \le [V_2,G]
= 1$ and similarly, $[V_1,G,G] = 1$. Hence, by the Three Subgroups
Lemma, we have $[G',V_1] = 1$. This implies that $V_1 \le C$.

Since $V_2 = [V_1,G]$, it follows that $V_1 \le A$, so we have $V_1
\le A \cap C$. If $V_1 < A \cap C$, then there exists an element $b
\in (A \cap C) \setminus V_1$, and we set $B/V_2 = \cent
{G/V_2}{bV_2}$. By Lemma \ref{Fact 6}, we know that $B \le C$. Now,
$[a,b] \in V_2$ since $b \in A$. This implies that $[b,a] \in V_2$,
and so, $a \in B$. Thus, $a \in C$ which is a contradiction, and so,
$A \cap C = V_1$.
\end{proof}

In several of the next lemmas, we will assume that $\card {G_3} =
p$. We will later see that for many of these results, this
hypothesis can be removed.  The first result determines some
conjugacy classes.  Notice that we do not need to assume anything
regarding $V_1$ for this lemma.  In fact, this next Lemma does not
impact $V_1$ at all.

\begin{lemma} \label{Fact 9}
Assume $G$ is nilpotent and $\card {G_3} = p$.  If $x \in G'
\setminus \gpcen G$, then ${\rm cl} (x) = xG_3$.
\end{lemma}

\begin{proof}
Because $G$ is nilpotent, $G = P \times Q$ where $P$ is a $p$-group
and $Q$ is a $p'$-group.  Hence, $G' = P' \times Q'$, and as $\card
{G_3} = p$, we have $G_3 = P_3$, so $Q_3 = 1$.  In particular, $Q'
\le \gpcen Q$, so $Q$ centralizes $G'$.  Observe that $G'/G_3$ is
central in $G/G_3$, and thus, it follows that ${\rm cl} (x)
\subseteq xG_3$.  We deduce that $\card {{\rm cl} (x)} \le p$.
Recall that $x \in G'$, which implies that $Q \le \cent Gx$. Now,
$\card {{\rm cl} (x)} = \gpindex G{\cent Gx}$ divides $\gpindex GQ =
\card P$. Therefore, $\card {{\rm cl} (x)}$ is either $1$ or $p$.
Because $x$ is not central, we must have $\card {{\rm cl} (x)} = p =
\card {xG_3}$, and we conclude that ${\rm cl} (x) = xG_3$.
\end{proof}

We now can strengthen the relationship between $m$ and $n$ in
$\gpindex G{V_1} = p^{2n}$ and $\gpindex {G'}{V_2} = p^m$ when
$\card {G_3} = p$.  In particular, we will show that $m = n$.  Also,
this shows that we have examples where $V_1 < C$ whenever $\gpindex
{G_3}{V_3} = p$.

\begin{lemma} \label{Fact 11}
Assume $V_3 = 1$ and $\card {G_3} = p$.  Then $\gpindex G{V_1} =
\gpindex {G'}{V_2}^2$ (i.e. $m = n$), $\gpindex GC = \gpindex
{G'}{V_2} = p^n$, and $V_2 = G' \cap \gpcen G$.
\end{lemma}

\begin{proof}
We know that $V_3 = 1$, so $V_2$ is central in $G$.  As $G_3 \le
V_2$, we have that $G_3$ is central.  Hence, $G$ is nilpotent with
nilpotence class $3$. Also, $C = \cent G{G'}$.  Let $Z = \gpcen G$.
It follows that $V_2 \le G' \cap Z$, and since $G'/V_2$ is a
$p$-group, we can set $\gpindex {G'}{G' \cap Z} = p^k$. Because $V_2
\le G' \cap Z$, we have $p^k \le \gpindex {G'}{V_2} = p^m$.  Now,
$G'/{V_2}$ is elementary abelian, so we can find $x_1, \dots, x_k
\in G' \setminus Z$ so that $G' = \langle x_1, \dots, x_k, G' \cap Z
\rangle$. It follows that $C = \bigcap\limits_{i=1}^k \cent G{x_i}$.
Thus, we have
$$
\gpindex GC = \gpindex G{\bigcap_{i=1}^k \cent G{x_i}} \le \prod_{i=1}^k
\gpindex G{\cent G{x_i}} = p^k,
$$
where the last equality follows from Lemma \ref{Fact 9}.

Since $G$ has nilpotence class $3$, $G'$ is not central in $G$, so
$C < G$. Thus, there is an element $a \in G \setminus C$, and let
$A/V_2 = \cent {G/V_2}{aV_2}$. By Lemma \ref{Fact 8}, we have $V_1 =
A \cap C$, so $\gpindex A{V_1} = \gpindex {AC}C$. Recall that
$\gpindex GA = p^m$ via Lemma \ref{Fact 7}.  Using Lemma \ref{VZquo}
and the previous paragraph, we know that
$$
p^{2m} \le p^{2n} = \gpindex G{V_1} = \gpindex GA \gpindex A{V_1} =
p^m \gpindex {AC}C \le p^m \gpindex GC \le p^mp^k \le p^{2m}.
$$
We must have equality throughout.  This implies that $m = n = k$ and
$\gpindex GC = p^n$. We obtain $\gpindex {G'}{G' \cap Z} = \gpindex
{G'}{V_2}$, and since $V_2 \le G' \cap Z$, this implies that $V_2 =
G' \cap Z$.
\end{proof}

To remove the hypothesis that $\card {G_3} = p$, we need the key
observation that if $V_3 \le N < G_3$, then $V_1 (G/N) = V_1/N$ by
Lemma \ref{quots}, and it follows that $V_2 (G/N) = [V_1/N,G/N] =
[V_1,G]N/N = V_2/N$. These follow because $G_3/V_3$ is central in
$G/V_3$, so that $N$ is normal. With these observations, we will be
able to use an inductive hypothesis on $G/N$. We now obtain
conclusion (1) in Theorem \ref{main3}.

\begin{lemma}\label{Fact 17}
Assume $V_3 < G_3$.  Then $\gpindex G{V_1} = \gpindex {G'}{V_2}^2$
(i.e. $m = n$).
\end{lemma}

\begin{proof}
When $|G_3| = p$, this is Lemma \ref{Fact 11}.  Thus, we may assume
that $|G_3| > p$.  We can find $V_3 \le N < G_3$ so that $\gpindex
{G_3}N = p$.  We apply Lemma \ref{Fact 11} to $G/N$.  Thus, we have
$\gpindex {G'/N}{V_2/N} = \gpindex G{V_2} = p^n$ and $\gpindex
G{V_1} = \gpindex {G/N}{V_1/N} = p^{2n}$.
\end{proof}

We now determine the structure of $C$ in the general case.  This is
conclusion (3) of Theorem \ref{main3}.

\begin{lemma}\label{Fact 17a}
Assume $V_3 < G_3$.  Then either $\gpindex GC = \gpindex C{V_1} =
\gpindex {G'}{V_2} = p^n$ or $C = V_1$.
\end{lemma}

\begin{proof}
We may assume that $V_3 = 1$.  We know that $G_3 > 1$, so $G$ has
nilpotence class $3$.  This implies that $G'$ is not central in $G$,
so $C < G$.  Thus, we can apply Lemma \ref{Fact 8} to see that $V_1
\le C$.  Also, since $G_3 > 1$, we can find $N \le G_3$ so that
$\gpindex {G_3}N = p$. Let $D/N = \cent {G/N}{G'/N}$. By Lemma
\ref{Fact 11}, we know that $\gpindex GD = \gpindex {G'}{V_2}$, and
by Lemma \ref{Fact 17}, we have $\gpindex G{V_1} = \gpindex
{G'}{V_2}^2 = \gpindex GD^2$.  It is easy to see that $C \le D$.  If
$C = D$, then we have $\gpindex GC = \gpindex C{V_1} = \gpindex
{G'}{V_2}$ as desired.

We now assume that $C < D$.  Thus, we can find $b \in D \setminus
C$.  Set $B/V_2 = \cent {G/V_2}{bV_2}$.  Applying Lemma \ref{Fact
8}, we have that $B \cap C = V_1$.  Since $V_1 \le C$, we have $b
\in D \setminus V_1$.  Thus, we may apply Lemma \ref{Fact 6} to see
that $B \le D$ and Lemma \ref{Fact 7} to see that $\gpindex GB =
\gpindex {G'}{V_2}$.  We have already seen that $\gpindex GD =
\gpindex {G'}{V_2}$.  It follows that $B = D$.  Hence, $C = D \cap C
= B \cap C = V_1$, and this implies the result.
\end{proof}

We now show that $V_2 = G' \cap Z$ where $Z/V_3 = \gpcen {G/V_3}$
when $V_3 < G_3$. This is conclusion (2) of Theorem \ref{main3}.

\begin{lemma}\label{Fact 18}
Assume $V_3 < G_3$.  If $Z/V_3 = \gpcen {G/V_3}$, then $V_2 = G'
\cap Z$.
\end{lemma}

\begin{proof}
Again, we can find a subgroup $N$ so that $V_3 \le N < G_3$ and
$\gpindex {G_3}N = p$.  We know $V_2 \le G'$ and $[V_2,G] = V_3$, so
$V_2 \le Z$.  It follows that $V_2 \le G' \cap Z$.  By Lemma
\ref{Fact 11}, we have $V_2/N = G'/N \cap \gpcen{G/N}$. Note that
$Z/N \le \gpcen {G/N}$. We now have
$$
\frac {G' \cap Z}N = \frac {G'}N \cap \frac {Z}N \le \frac {G'}N
\cap \gpcen {\frac GN} = \frac {V_2}N \le \frac {G' \cap Z}N.
$$
Since we must have equality throughout, we obtain $V_2/N = (G' \cap
Z)/N$, and therefore, $G' \cap Z = V_2$.
\end{proof}

We obtain an equality of centralizers when $\card {G_3} = p$.

\begin{lemma}\label{Fact 12}
Assume $V_3 = 1$ and $\card {G_3} = p$.  If $b \in C \setminus V_1$,
then $C/V_2 = \cent {G/V_2}{bV_2}$.
\end{lemma}

\begin{proof}
Let $B/V_2 = \cent {G/V_2}{bV_2}$.  By Lemma \ref{Fact 7}, we have
$\gpindex GB = p^n$.  On the other hand, we proved in Lemma
\ref{Fact 11} that $\gpindex GC = p^n$.  Finally, in Lemma \ref{Fact
6}, we showed that $B \le C$.  Hence, $B = C$, and the result is
proved.
\end{proof}

We find that $(C,V_1)$ is a GCP when $V_3 = 1$ and $\card {G_3} =
p$.

\begin{lemma}\label{Fact 13}
Assume $V_3 = 1$ and $\card {G_3} = p$.  Then $C' = V_2$, $C$ has
nilpotence class $2$, and $(C,V_1)$ is a GCP.
\end{lemma}

\begin{proof}
First, notice that $C' \le G' \le \gpcen C$, so $C$ has nilpotence
class $2$.  Consider $[C',G] = [C,C,G]$. Observe that $[C,G,C] \le
[G',C] = 1$, and similarly, $[G,C,C] = 1$. By the Three Subgroups
Lemma, we deduce that $[C',G] = 1$, and thus, $C' \le Z_1$.  We also
have $C' \le G'$, and hence, $C' \le Z_1 \cap G'$. By Lemma
\ref{Fact 11}, this implies that $C' \le V_2$.

By Lemma \ref{Fact 11}, we have $\gpindex G{V_1} = p^{2n}$ and
$\gpindex GC = p^n$, so $\gpindex C{V_1} = p^n$.  Hence, we can find
an element $b \in C \setminus V_1$.  Using Lemma \ref{Fact 12}, we
have $C = \{ x \in G \mid [b,x] \in V_2 \}$.  Since $(G,V_1)$ is a
GCP and $b \not\in V_1$, for each $y \in V_2$, there is an element
$g \in G$ so that $[b,g] = y$.  We have just seen that this implies
that $g \in C$, and $V_2 \le [b,C] \le C'$.  We conclude that $C' =
V_2$ and $(C,V_1)$ is a GCP.
\end{proof}

We get the general version of Lemma \ref{Fact 13}.  If $V_1 = C$,
then obviously, $V (C) \le V_1$.  Hence, to prove conclusion (4) of
Theorem \ref{main3}, it suffices to prove that $V(C) \le V_1$ when
$V_1 < C$.  In the next lemma, we prove that $(C,V_1)$ is a GCP when
$V_1 < C$, and from Lemma \ref{V(G)} implies that $V (C) \le V (G)$.
We also obtain conclusion (5) of Theorem \ref{main3}.

\begin{lemma}\label{Fact 20}
Assume $V_3 < G_3$.  If $V_1 < C$, then $C' = V_2$, $C/V_3$ has
nilpotence class $2$, and $(C,V_1)$ is a GCP.
\end{lemma}

\begin{proof}
First, notice that $C' \le G'$ and $G'/V_3 \le \gpcen {C/V_3}$, so
$C/V_3$ has nilpotence class $2$.  Consider a subgroup $V_3 \le N <
G_3$ so that $\gpindex {G_3}N = p$.  We know that $\gpindex
{G/N}{\cent {G/N}{G'/N}} = p^n$ via Lemma \ref{Fact 11}.  In light
of Lemma \ref{Fact 17a}, we have $\gpindex GC = p^n$.  Since $C/N
\le \cent {G/N}{G'/N}$, this implies that $C/N = \cent {G/N}{G'/N}$.
By Lemma \ref{Fact 13}, we have $C'N/N = V_2/N$, and hence, $C'N =
V_2$. {}From Lemma \ref{Fact 4}, we know that $G_3/V_3$ is a
elementary abelian $p$-group, so we can find subgroups $N_1, \dots,
N_r$ in $G_3$ containing $V_3$ so that $\gpindex {G_3}{N_i} = p$ for
each $i$ and $\bigcap\limits_{i=1}^r N_i = V_3$. We have $C'N_i =
V_2$ for each $i$. This implies that
$$
V_2 = \bigcap_{i=1}^r (C'N_i) \ge C' \left( \bigcap_{i=1}^r N_i
\right) = C' V_3,
$$
and so $C' \le V_2$.

We know via Lemma \ref{Fact 17} that $\gpindex G{V_1} = p^{2n}$.
Since $\gpindex GC = p^n$, we have $V_1 < C$.  Hence, we can find $b
\in C \setminus V_1$, and let $B/V_2 = \cent {G/V_2}{bV_2}$.  By
Lemma \ref{Fact 6}, we have $B \le C$.  Since $b \not\in V_1$ and
$(G,V_1)$ is a GCP, there is for each $y \in V_2$ an element $g \in
G$ so that $[b,g] = y$. Notice that this implies that $g \in B \le
C$.  We now have $V_2 \le [b,C] \le C'$.  We conclude that $V_2 =
C'$ and $(C,V_1)$ is a GCP.
\end{proof}

Finally, we obtain Conclusion (6) of Theorem \ref{main3}.

\begin{lemma}\label{Fact 20a}
Assume $V_3 < G_3$.  If $V_1 < C$ and $[V_1,C] < V_2$, then
$C/[V_1,C]$ is a VZ-group and $n$ is even.
\end{lemma}

\begin{proof}
By Lemma \ref{Fact 20}, we know that $(C,V_1)$ is a GCP.  It follows
that $V (C) \le V_1$.  {}From Lemma \ref{basics}, $\gpcen
{C/[V_1,C]} \le V(C)/[V_1,C] \le V_1/[V_1,C] \le \gpcen
{C/[V_1,C]}$.  We deduce that $V(C) = V_1$. Now, $V_2 (C) = [V_1,C]
< V_2 = C'$ using Lemma \ref{Fact 20}. Applying Lemma \ref{VZquo},
$C/[V_1,C]$ is a VZ-group. Since $\gpindex C{V_1} = p^n$ via Lemma
\ref{Fact 17a}, we deduce that $n$ is even.
\end{proof}

\section{Examples}

We now present some examples to demonstrate the phenomenon we have
mentioned.  Many of the examples we present were found using the
small groups library in MAGMA (see \cite{MAGMA}) or GAP (see
\cite{GAP}).  For such groups, we will include their identification
in the small groups library.  We should note that the examples we
present are in no way unique and that one can find many similar
examples by exploring the small groups libraries.  We would expect
that it is possible to find constructions that work for any prime
$p$ for the groups in many of these examples.  All the computations
in this section were done using MAGMA or GAP, and so, we shall not
explicitly state this.

We begin with an example that shows that $V (G)$ is not necessarily
proper in $G$ when $G$ is solvable (nilpotent or not nilpotent). The
easiest way to construct such an example is as follows. Let $H$ and
$K$ be any pair of nonabelian groups (possibly isomorphic; they may
or may not be nilpotent), and let $G = H \times K$. We can find
nonlinear characters $\chi, \psi \in \irr G$ so that $H \le \ker
{\chi}$ and $K \le \ker {\psi}$. Notice that this implies that $H
\le V (\chi)$ and $K \le V (\psi)$, and hence, $G = HK \le V (G)$,
and so, $G = V (G)$.

Looking at the small groups library, we see that $V(G) < G$ for
every nonabelian $2$-group whose order is $8$, $16$, or $32$.
Suppose $G = A \times B$ where each of $A$ and $B$ are either
quaternion or dihedral of order $8$, then $G = V(G)$, so there exist
groups of order $64$ for which the vanishing-off subgroup is not
proper.

We now present an example of a group $G$ of order $64$ for which $G
= V(G)$, but $G$ is not a direct product of nonabelian groups.  We
set
$$
G = \langle a,b,c,d,e,f | a^2= b^2 = c^2 = d^2 = e^2 = f^2 = 1, b^a
= bd, c^a = ce, c^b = cf \rangle
$$
where we follow the conventions of a polycyclic representation in
GAP or MAGMA that commutators of generators that are not given must
be trivial.  (This is SmallGroup (64,73).)  We computed that $V (G)
= G$. We also computed the character degrees of $G$ which were $\{
1, 2 \}$, so $G$ cannot be the direct product of nonabelian groups.

We next include an example where $G = V(G)$ and $G$ has Camina
elements. Again, $G$ has order $64$.  We set
$$
G = \langle a,b,c,d,e,f | a^2= d, c^2 = f, b^2 = d^2 = e^2 = f^2 =
1, b^a = bc, c^a = ce, c^b = cf, d^b= def \rangle
$$
and we have $G = V(G)$. (This is SmallGroup (64,8).)  Using MAGMA,
one can see that $b$ and $bd$ are Camina elements of $G$.

We now present some examples of $p$-groups $G$ where $\gpindex
G{V(G)}$ is not a square.  The first example is the dihedral (or
semi-dihedral) group of order $2^n$ with $n \ge 4$. In this case, $V
(G)$ is the cyclic subgroup of order $2^{n-1}$ which of course has
index $2$ in $G$.   

Next, we present an example of a group of order $81$ with a
vanishing-off subgroup of index $3$.  In this example, we have
$$
G = \langle a,b,c,d | a^3 = b^3 = c^3 = d^3 = 1, b^a = bc, c^a = cd
\rangle
$$
and $V (G) = \langle b, c, d \rangle$ which has index $3$.  (This is
SmallGroup (81,7).)

We also present an example of a group of order $256$ where $\gpindex
G{V (G)} = 8$.  Also, this is an example where $V (G) < G$ and $V
(G)$ contains a Camina element for $G$.  In this case, we have
$$
G = \langle a,b,c,d,e,f,g,h | a^2 = g, c^2 = de, d^2 = e, f^2 = h,
b^2 = e^2 = g^2 = h^2 = 1, b^a = be, $$ $$c^a = cf, c^b = cef, d^a =
def, d^b = de, e^a = eh, f^b = fh, g^b = gh, g^c = gh \rangle.
$$
(This is SmallGroup (256,16349).)  We compute that $V (G) = \langle
d, e, f, g, h \rangle$, so $\gpindex G{V (G)} = 8$.  We should also
note that while every element in $G \setminus V (G)$ is a Camina
element for $G$, this is an example where not all the Camina
elements for $G$ lie in $G \setminus V(G)$.  In particular, $V (G)$
will itself contain a Camina element for $G$. We identify $dg$ as a
Camina element for $G$ and $dg \in V (G)$.

We next present an example of a group $G$ with vanishing height $3$
where $\gpindex {G'}{V_2 (G)}$ is not a square.  In this case, we
have $\card G = 32$.  Take
$$
G = \langle a,b,c,d,e | a^2 = d, b^2 = c^2 = d^2 = e^2 = 1, b^a =
bc, c^a = ce, d^b = de \rangle.
$$ (This is SmallGroup (32,6).)  In this case, $V (G) = \langle c,
d, e \rangle$ so $V (G)$ has index $4$.  Also, $G' = \langle c, e
\rangle$ and $V_2 (G) = \langle e \rangle$, so $\gpindex {G'}{V_2
(G)} = 2$.  Finally, to see that $G$ has vanishing height $3$, we
note that $G_3 = \langle e \rangle$ and $V_3 (G) = 1$.

This next example we believe is instructive.  It is an example where
$\gpindex G{V(G)} = \gpindex {G_3}{V_3 (G)} = 4$ but $\gpindex
{G'}{V_2 (G)} = 2$.  We also deduce that $C = V_1 (G)$.  In the
examples we have looked at, it seems that $C = V_1 (G)$ occurs
exactly when $\gpindex G{V(G)} = \gpindex {G_3}{V_3 (G)}$, but we
have not been able to prove this.  In this case, we have $\card G =
128$, and we take
$$
G = \langle a,b,c,d,e,f,g | a^2 = d, b^2 = e, c^2 = d^2 = e^2 = f^2
= g^2 = 1, b^a = bc, c^a = cf, c^b = cg,
$$
$$
d^b = df, e^a = eg \rangle.
$$
(This is SmallGroup (128,36).)  We deduce that $V (G) = \langle
c,d,e,f,g \rangle = \cent G{G'} = C$. We also have that $G' =
\langle c,f,g \rangle$ and $V_2 (G) = \langle f,g \rangle$.
Finally, $G_3 = \langle f,g \rangle$ and $V_3 (G) = 1$.

We now present an example with vanishing height $4$.  This example
will also have nilpotence class $4$, and $Z_3 (G) = V(G)$.  We have
a group $G$ with $\card G = 128$. We take
$$
G = \langle a,b,c,d,e,f,g | a^2 = e, c^2 = g, b^2 = d^2 = e^2 = f^2
= g^2 = 1, b^a = bd, c^b = cg, $$ $$d^a = df, e^b = ef, e^d = eg,
f^a = fg \rangle.
$$
(This is SmallGroup (128,854).)  In this case, $G_4 = \langle g
\rangle$ and $V_4 (G) = 1$.  Observe that $G$ has nilpotence class
$4$ and $V_4 (G) < G_4$, so $G$ has vanishing height $4$.  In fact,
$V (G) = Z_3 (G) = \langle c,d,e,f,g \rangle$.  We note that there
are no groups of order $256$ with vanishing height $5$.

For our final example, we present a group of order $3^7$ with
nilpotence class and vanishing height $4$ with $Z_3 (G) = V(G)$.  In
this case, we take
$$
G = \langle a,b,c,d,e,f,g | a^3 = d, b^3 = c, c^3 = e, d^3 = f, e^3
= g, b^a = bc, c^a = ce,$$ $$ d^b = de^2, d^c = dg^2, e^a = eg, f^b
= fg^2 \rangle.
$$
(This is SmallGroup ($3^7$, 194).)  We have $V (G) = Z_3 (G) =
\langle c,d,e,f,g \rangle$. This implies that $G$ has nilpotence
class and vanishing height $4$ as desired.

\medskip
{\bf Acknowledgement:}  We would like to thank Mohammad Jafari for a
careful reading of this paper, for pointing out a number of typos,
and particularly for pointing out a mistake in an earlier version of
this note so that we could correct the mistake.

\end{document}